\documentclass{amsart}
\usepackage{amssymb}
\usepackage{amsfonts}
\usepackage{amssymb}
\usepackage{amsmath}
\usepackage{amsthm}
\usepackage{enumerate}
\usepackage{tabularx}
\usepackage{centernot}
\usepackage{mathtools}
\usepackage{stmaryrd}
\usepackage{amsthm,amssymb}
\usepackage{etoolbox}
\usepackage{tikz}
\usepackage{amssymb}
\usetikzlibrary{matrix}
\usepackage{tikz-cd}
\usepackage{tikz}

\renewcommand{\leq}{\leqslant}
\renewcommand{\geq}{\geqslant}

\makeatletter
\def\subsection{\@startsection{subsection}{3}%
  \z@{.5\linespacing\@plus.7\linespacing}{.3\linespacing}%
  {\bfseries\centering}}
\makeatother

\makeatletter
\def\subsubsection{\@startsection{subsubsection}{3}%
  \z@{.5\linespacing\@plus.7\linespacing}{.3\linespacing}%
  {\centering}}
\makeatother

\makeatletter
\def\myfnt{\ifx\protect\@typeset@protect\expandafter\footnote\else\expandafter\@gobble\fi}
\makeatother

\newtheorem{theorem}{Theorem}

\newtheorem{corollary}[theorem]{Corollary}
\newtheorem{definition}[theorem]{Definition}

\newtheorem{fact}[theorem]{Fact}

\newcounter{claimcounter}
\numberwithin{claimcounter}{theorem}
\newenvironment{claim}{\stepcounter{claimcounter}{\noindent {\underline{\em Claim \theclaimcounter}.}}}{}
\newenvironment{claimproof}[1]{\noindent{{\em Proof.}}\space#1}{\hfill $\rule{0.40em}{0.40em}$}

\newcommand{\pureindep}[1][]{%
  \mathrel{
    \mathop{
      \vcenter{
        \hbox{\oalign{\noalign{\kern-.3ex}\hfil$\vert$\hfil\cr
              \noalign{\kern-.7ex}
              $\smile$\cr\noalign{\kern-.3ex}}}
      }
    }\displaylimits_{#1}
  }
}

\newcommand{\indep}[2]{%
  \mathrel{
    \mathop{
      \vcenter{
        \hbox{%
\oalign{
\noalign{\kern-.3ex}\hfil$\vert$\hfil\cr
              \noalign{\kern-.7ex}
              $\smile$\cr\noalign{\kern-.3ex}
}
}
      }
}^{\!\!\!\!\!#2}_{\!\!\hspace{-0.1em}#1}
  }
}

\newcommand{\displayindep}[2]{%
  \mathrel{
    \mathop{
      \vcenter{
        \hbox{%
\oalign{
\noalign{\kern-.3ex}\hfil$\vert$\hfil\cr
              \noalign{\kern-.7ex}
              $\smile$\cr\noalign{\kern-.3ex}
}
}
      }
}^{\!\!\hspace{-0.1em}#2}_{\!\!\hspace{-0.1em}#1}
  }
}

\newcommand{\displayfindep}[2]{%
  \mathrel{
    \mathop{
      \vcenter{
        \hbox{%
\oalign{
\noalign{\kern-.3ex}\hfil$\vert$\hfil\cr
              \noalign{\kern-.7ex}
              $\smile$\cr\noalign{\kern-.3ex} 
}
}
      }
}^{\!\hspace{-0.14em}#2}_{\!\!\hspace{-0.05em}#1}
  }
}

\begin{document}

\begin{abstract} We prove that no uncountable Polish group can admit a system of generators whose associated length function satisfies the following conditions:
\begin{enumerate}[(i)]
\item if $0 < k < \omega$, then $lg(x) \leq lg(x^k)$;
\item if $lg(y) < k < \omega$ and $x^k = y$, then $x = e$.
\end{enumerate}
In particular, the automorphism group of a countable structure cannot be an uncountable right-angled Artin group. This generalizes results from \cite{shelah} and \cite{solecki}, where this is proved for free and free abelian uncountable groups.
\end{abstract}

\title{No Uncountable Polish Group Can be a Right-Angled Artin Group}
\thanks{Partially supported by European Research Council grant 338821. No. 1112 on Shelah's publication list.}

\author{Gianluca Paolini}
\address{Einstein Institute of Mathematics,  The Hebrew University of Jerusalem, Israel}

\author{Saharon Shelah}
\address{Einstein Institute of Mathematics,  The Hebrew University of Jerusalem, Israel \and Department of Mathematics,  Rutgers University, U.S.A.}

\maketitle


	In a meeting in Durham in 1997, Evans asked if an uncountable free group can be realized as the group of automorphisms of a countable structure. This was settled in the negative by Shelah \cite{shelah}. Independently, in the context of descriptive set theory, Becher and Kechris \cite {kechris} asked if an uncountable Polish group can be free. This was also answered negatively by Shelah \cite{shelah_1}, generalizing the techniques of \cite{shelah}. Inspired by the question of Becher and Kechris, Solecki \cite{solecki} proved that no uncountable Polish group can be free abelian. In this paper we give a general framework for these results, proving that no uncountable Polish group can be a right-angled Artin group (see below for a definition). We actually prove more:
	
	\begin{theorem}\label{main_th} Let $G = (G, \mathfrak{d})$ be an uncountable Polish group and $A$ a group admitting a system of generators whose associated length function satisfies the following conditions:
\begin{enumerate}[(i)]
\item if $0 < k < \omega$, then $lg(x) \leq lg(x^k)$;
\item if $lg(y) < k < \omega$ and $x^k = y$, then $x = e$.
\end{enumerate}
Then $G$ is not isomorphic to $A$, in fact there exists a subgroup $G^*$ of $G$ of size $\mathfrak{b}$ (the bounding number) such that $G^*$ is not embeddable in $A$.
\end{theorem}

	\begin{proof} Let $\zeta = (\zeta_n)_{n < \omega} \in \mathbb{R}^\omega$ be such that $\zeta_n < 2^{-n}$, for every $n < \omega$, and $\bar{g} = (g_n)_{n < \omega} \in G^{\omega}$ such that $g_n \neq e$ and $\mathfrak{d}(g_n, e) < \zeta_n$, for every $n < \omega$. Let $\Lambda$ be a set of power $\mathfrak{b}$ of increasing functions $\eta \in \omega^\omega$ which is unbounded with respect to the partial order of eventual domination. For transparency we also assume that for every $\eta \in \Lambda$ we have $\eta(0) > 0$. For $\eta \in \Lambda$, define the following set of equations:
	$$ \Gamma_{\eta} = \{ x_{n+1}^{\eta(n)} = x_ng_n : n < \omega \}.$$
By \cite{shelah_1}, for every $\eta \in \Lambda$, $\Gamma_{\eta}$ is solvable in $G$. Let $\bar{b}_{\eta} = (b_{\eta, n})_{n < \omega}$ witness it, i.e.:
	$$ \bar{b}_{\eta} \in G^{\omega} \;\; \text{ and } \;\; \bigwedge_{n < \omega} b_{\eta, n+1}^{\eta(n)} = b_{\eta, n}g_n.$$
Let $G^*$ be the subgroup of $G$ generated by $\{ g_n : n < \omega \} \cup \{ b_{\eta, n} : \eta \in \Lambda, n < \omega \}$. Towards contradiction, suppose that $\pi$ is an embedding of $G^*$ into $A$, and let $S$ be a system of generators for $A$ whose associated length function $lg_S = lg$ satisfies conditions (i) and (ii) of the statement of the theorem. For $\eta \in \Lambda$ and $n < \omega$, let:
$$\pi(g_n) = g'_n, \;\; \pi(b_{\eta, n}) = c_{\eta, n} \;\; \text{ and } \;\; m_*(\eta) = lg(c_{\eta, 0}).$$ 
Now, $m_*$ is a function from $\Lambda$ to $\omega$ and so there exists unbounded $\Lambda_1 \subseteq \Lambda$ such that for every $\eta \in \Lambda_1$ the value $m_*(\eta)$ is a constant $m_*$. Fix such a $\Lambda_1$ and $m_*$, and let $f_1, f_2 \in \omega^{\omega}$ increasing satisfying the following:
\begin{enumerate}[(1)]
\item $f_1(n) > lg(g'_n)$;
\item $f_2(n) = (m_* +1) + \sum_{\ell < n} f_1(\ell)$.
\end{enumerate}

\begin{claim} For every $\eta \in \Lambda_1$, $lg(c_{\eta, n}) < f_2(n)$.
\end{claim}

\begin{claimproof} By induction on $n < \omega$. The case $n = 0$ is clear by the choice of $f_1$ and $f_2$. Let $n = m+1$. Because of assumption (i) on $A$, the choice of $\Lambda_1$ and the choice of $f_1$ and $f_2$, we have:
\[ \begin{array}{rcl}
	lg(c_{\eta, n}) & \leq & lg(c_{\eta, n}^{\eta(m)}) \\
					  & =    & lg(c_{\eta, m}g'_m) \\
					  & \leq & lg(c_{\eta, m}) + lg(g'_m) \\
					  & <    & f_2(m) + f_1(m) \\
					  & =    & f_2(n). 
\end{array}	\]
\end{claimproof}

\noindent Now, by the choice of $\Lambda_1$, we can find $\eta \in \Lambda_1$ and $n < \omega$ such that $\eta(n) > f_2(n+2)$. Notice then that by the claim above and the choice of $f_1$ and $f_2$ we have:
\begin{equation}\label{star} 
\eta(n) > f_2(n+1) = f_2(n) + f_1(n) > lg(c_{\eta, n}) + lg(g'_n) \geq lg(c_{\eta, n}g'_n),
\end{equation}
\begin{equation}\label{starstar} 
\eta(n) > f_2(n + 2) \geq f_1(n +1) > lg(g'_{n+1}).
\end{equation}
Thus, by (1) and the fact that $c_{\eta, n+1}^{\eta(n)} = c_{\eta, n}g'_n$, using assumption (ii) we infer that $c_{\eta, n+1} = e$. Hence,
$$c_{\eta, n+2}^{\eta(n+1)} = c_{\eta, n+1}g'_{n+1} = g'_{n+1}.$$
Furthermore, if $\eta(n+1) > lg(g'_{n+1})$, then, again by assumption (ii), we have that $c_{\eta, n+2} = e$, and so $c_{\eta, n+2}^{\eta(n+1)} = g'_{n+1} = e$, which contradicts the choice of $(g_n)_{n < \omega}$. Hence, $\eta(n) < \eta(n+1) \leq lg(g'_{n+1})$, contradicting (2). It follows that the embedding $\pi$ from $G^*$ into $A$ cannot exist.
\end{proof}

	\begin{definition} Given a graph $\Gamma = (E, V)$, the {\em right-angled Artin group} $A(\Gamma)$ is the group with presentation $\langle V \mid ab = ba : a E b \rangle$. 
\end{definition}

Thus, for $\Gamma$ a graph with no edges (resp. a complete graph) $A(\Gamma)$ is a free group (resp. a free abelian group). 


	\begin{definition}  Let $A(\Gamma)$ be a right-angled Artin group and $lg$ its associated length function. We say that an element $g  \in A(\Gamma)$ is cyclically reduced if it cannot be written as $g = h f h^{-1}$ with $lg(g) = lg(f) + 2$.
\end{definition}

	\begin{fact}\label{new_div} Let $A(\Gamma)$ be a right-angled Artin group, $lg$ its associated length function and $g  \in A(\Gamma)$. Then:
	\begin{enumerate}[(1)]
	\item $g$ can be written as $hfh^{-1}$ with $f$ cyclically reduced and $lg(g) = lg(f) + 2 lg(h)$;
	\item if $0 < k < \omega$ and $f$ is cyclically reduced, then  $lg(f^k) = klg(f)$;
	\item if $0 < k < \omega$ and $g = hfh^{-1}$ is as in (1), then $lg(hfh^{-1})^k = klg(f) + 2lg(h)$.
\end{enumerate}
\end{fact}

	\begin{proof} Item (1) is proved in \cite[Proposition on pg.38]{servatius}. The rest is folklore.
\end{proof}

	\begin{corollary} No uncountable Polish group can be a right-angled Artin group.
\end{corollary}

	\begin{proof} By Theorem \ref{main_th} it suffices to show that for every right-angled Artin group $A(\Gamma)$ the associated length function $lg$ satisfies conditions (i) and (ii) of the theorem, but by Fact \ref{new_div}  this is clear. 
\end{proof}

	As well known, the automorphism group of a countable structure is naturally endowed with a Polish topology which respects the group structure, hence:

	\begin{corollary} The automorphism group of a countable structure can not be an uncountable right-angled Artin group.
\end{corollary}

	The situation is different for right-angled Coxeter groups, in fact the structure $M$ with $\omega$ many disjoint unary predicates of size $2$ is such that $Aut(M) = (\mathbb{Z}_2)^\omega$, i.e. $Aut(M)$ is the right-angled Coxeter group on $K_{\mathfrak{c}}$ (a complete graph on continuum many vertices). Notice that in this group for any $a \neq b \in K_{\mathfrak{c}}$ we have:
	\begin{enumerate}[(i)]
	\item $(ab)^2 = 1$;
	\item $lg(ab) = 2 < 3$, $(ab)^3 = ab$ and $ab \neq e$.
	\end{enumerate}
We hope to investigate realizability of uncountable right-angled Coxeter groups as groups of automorphisms of countable structures in a future work.


\begin{thebibliography}{10}

\bibitem{kechris}
Howard Becker and Alexander S. Kechris.
\newblock {\em The Descriptive Set Theory of Polish Group Actions}.
\newblock London Math. Soc. Lecture Notes Ser. 232, Cambridge University Press, 1996. 


\bibitem{servatius}
Herman Servatius.
\newblock {\em Automorphisms of Graph Groups}.
\newblock J. Algebra {\bf 126} (1989), 34-60.

\bibitem{shelah}
Saharon Shelah.
\newblock {\em A Countable Structure Does Not Have a Free Uncountable Automorphism Group}.
\newblock Bull. London Math. Soc. {\bf 35} (2003), 1-7.

\bibitem{shelah_1}
Saharon Shelah.
\newblock {\em Polish Algebras, Shy From Freedom}.
\newblock Israel J. Math. {\bf 181} (2011), 477-507.

\bibitem{solecki}
S{\l}awomir Solecki.
\newblock {\em Polish Group Topologies}.
\newblock In: Sets and Proofs, London Math. Soc. Lecture Note Ser. 258. Cambridge University Press, 1999.

\end{thebibliography}
\end{document}